\documentclass[11pt]{article}
\usepackage{amssymb,amsmath,latexsym}
\usepackage{amsthm}
\usepackage[shortlabels]{enumitem}
\usepackage{caption}
\usepackage{xcolor}
\usepackage{float}
\usepackage{graphicx}

\usepackage{url}
\usepackage{breakurl}
\usepackage[colorlinks=true,citecolor=blue,linkcolor=blue,urlcolor=blue,bookmarks,bookmarksopen,bookmarksdepth=2,backref=page,breaklinks]{hyperref}

\usepackage{bookmark}
\bookmarksetup{
	numbered, 
	open,
}

\newcommand*{\doi}[1]{\href{https://doi.org/\detokenize{#1}}{https://doi.org/\detokenize{#1}}}
\renewcommand*{\backref}[1]{}
\renewcommand*{\backrefalt}[4]{%
	\ifcase #1 (Not cited.)%
	\or        (Cited on page~#2.)%
	\else      (Cited on pages~#2.)%
	\fi}

\newcommand{\Z}{\mathbb Z}

\newcommand{\F}{\mathbb F}

\DeclareMathOperator{\Nb}{Nb}
\DeclareMathOperator{\Tr}{Tr}

\theoremstyle{plain}
\newtheorem{theorem}{Theorem}[section]
\newtheorem{lemma}[theorem]{Lemma}
\newtheorem{proposition}[theorem]{Proposition}

\newtheorem{corollary}[theorem]{Corollary}

\theoremstyle{definition}

\newtheorem{definition}[theorem]{Definition}
\newtheorem{remark}[theorem]{Remark}

\numberwithin{theorem}{section}
\numberwithin{equation}{section}
\numberwithin{table}{section}
\numberwithin{figure}{section}

\DeclareMathOperator{\image}{Im}
\begin{document}
\title{The combinatorial structure and value distributions of plateaued functions}

\author{Lukas K\"olsch$^1$ and Alexandr Polujan$^2$ \vspace{0.4cm} \ \\
$^1$ University of South Florida \\\tt lukas.koelsch.math@gmail.com\vspace{0.4cm}\\
$^2$ Otto-von-Guericke-Universit\"{a}t Magdeburg,\ \\ Universit\"{a}tsplatz 2, 39106, Magdeburg, Germany\ \\ \tt alexandr.polujan@ovgu.de
}

\date{\today}
\maketitle
\abstract{

We study combinatorial properties of plateaued functions $F \colon \mathbb{F}_p^n \rightarrow \mathbb{F}_p^m$. All quadratic functions, bent functions and most known APN functions are plateaued, so many cryptographic primitives rely on plateaued functions as building blocks. The main focus of our study is the interplay of the Walsh transform and linearity of a plateaued function, its differential properties, and their value distributions, i.e., the sizes of image and preimage sets. In particular, we study the special case of ``almost balanced'' plateaued functions, which only have two nonzero preimage set sizes, generalising, for instance, all monomial functions. We achieve several direct connections and (non)existence conditions for these functions, showing in particular that plateaued $d$-to-$1$ functions (and thus plateaued monomials) only exist for a very select choice of $d$, and we derive for all these functions their linearity as well as bounds on their differential uniformity. We also specifically study the Walsh transform of plateaued APN functions and their relation to their value distribution. \\[1mm]

\noindent\textbf{Keywords:} Plateaued function, APN function, Value distribution, Differential uniformity, Monomial function. \ \\


\thispagestyle{empty}

\section{Introduction}

Plateaued functions can be seen as a generalization of quadratic functions which play an important role in cryptography since they may have several desirable cryptographic properties such as high nonlinearity, resiliency, and low additive autocorrelation, to name a few.  More specifically, the set of plateaued functions forms a superclass of several classes of functions that are used as cryptographic primitives in a variety of contexts, namely linear functions, quadratic functions, bent functions, and most of the known examples and constructions of APN functions. All of these classes of functions have attracted a lot of attention from the research community due to their practical and theoretical importance. For example, the only known instance of APN permutations in an even number of variables (the existence of which is called the ``Big APN'' problem) was obtained from a quadratic (and hence plateaued) APN function~\cite{APNPermutation}. Bent functions are extremal combinatorial objects that are at the maximum distance from all affine functions. They have been recently used in the design of block ciphers~\cite{Wiemer20}. Quadratic mappings have been used as crucial building blocks of cryptographic primitives, for instance, the quadratic mapping $\chi$ is an important primitive used in various cryptographic algorithms like \textsc{Keccak}, \textsc{Ascon}, and \textsc{Xoodoo}; the study of its algebraic properties has recently been the subject of much research, see e.g.~\cite{KK2024,Schoone2024}.

Since the class of plateaued functions essentially generalizes the classes of quadratic and bent functions and contains many APN functions, it is essential to investigate further the natural properties of these functions within the general plateaued class as well. One of the current research directions in the analysis of cryptographically significant (vectorial) functions involves investigating the value distributions of such mappings. Of particular interest are the questions of whether a given class of mappings includes balanced functions (especially permutations), and if not, how far they deviate from being balanced; the latter can be measured by the value called imbalance~\cite{carlet2007nonlinearities}. For example, the problem of bijectivity of the mapping $\chi$ was resolved recently, see~\cite{AMG2024,Schoone2024}.  Value distributions of APN functions have been thoroughly investigated in~\cite{Koelsch2023}, while the case of bent functions was addressed in~\cite{KoelschPolujan2024}. As these studies indicate, in many cases, such functions can be almost balanced, meaning that all preimages, except one, have the same size. In turn, such functions can be considered, in some sense, the second best from the value distribution point of view. As recently demonstrated in~\cite{Koelsch2023,KoelschPolujan2024} and~\cite{AycyaWilfriedIEEE2025}, relying on the analysis of value distributions, it is possible to derive information about the combinatorial structure of the considered functions. 

The main aim of this paper is to investigate the value distributions of plateaued functions and provide a detailed combinatorial structure of these functions by developing new theoretical tools connecting the Walsh transform and imbalance. These results are then connected to cryptographic properties like nonlinearity and differential uniformity. \\

\noindent\textbf{Our contribution.} Firstly, we introduce general bounds on preimage set sizes of a mapping $F\colon\F_p^n\to\F_p^m$ in terms of its imbalance and the size of the image set (Theorem~\ref{th:cs_general}). With these bounds, we define almost balanced functions as the functions that achieve equality and have a nonzero imbalance (Definition~\ref{def: AB}). Thus, we extend the notion of almost balanced functions~\cite{KoelschPolujan2024} to a non-surjective case using spectral properties of the considered functions instead of differential ones. In this way, we complement the original results obtained in~\cite{KoelschPolujan2024}.

Secondly, we develop further connections between imbalance, value distributions, and the Walsh transform of functions $F\colon\F_p^n\to\F_p^m$. More precisely, we derive strong connections between the Walsh transform and the imbalance of almost balanced surjective functions (Theorem~\ref{thm:AB}). We also show that every function $F\colon\F_p^n\to\F_p^m$ with $m\le n/2$ is equivalent to a surjective function (Theorem~\ref{th: everything is surjective}) in the sense that there is an additive mapping $L$ such that $F+L$ is surjective.

Then, we concentrate on the combinatorial structure of plateaued functions. We consider $d$-to-$1$ functions which are special almost balanced functions $F \colon \F_p^n \rightarrow \F_p^m$ with the property: $F$ has exactly one preimage set of size $1$, and all other nonempty preimage sets have size exactly $d$. We show that $d$-to-$1$ vectorial functions can only be plateaued for very specific values of $d$ and derive precise information on the amplitudes of the components (Theorem~\ref{thm:platdto1}). 
With these results, we further concentrate on the analysis of plateaued monomials noting that monomial functions are important arithmetization-oriented primitives that are used in many hash functions used in zero-knowledge proof systems, like \textsc{Poseidon}~\cite{grassi2021poseidon} and \textsc{Rescue-Prime}~\cite{szepieniec2020rescue}. Specifically, we determine the possible shape of plateaued monomials $x^d$ on $\F_{p^n}$ (Corollary~\ref{cor:monom}) and show that plateaued mappings induced by monomials are either permutations or $(p^t+1)$-to-1 mappings. Additionally, we investigate differential properties of such functions (Theorems~\ref{thm:diff} and~\ref{thm:gcd1}) and provide explicit constructions of surjective almost balanced plateaued functions: monomials (Theorem~\ref{thm:gold}), and Maiorana-McFarland functions (Propositions~\ref{prop: MM1} and~\ref{prop: MM2}).

Lastly, we investigate the components of plateaued APN functions. As a highlight, we determine the combinatorial structure of plateaued APN functions with extremal values of imbalance (Proposition~\ref{prop: APNs with extreme imbalance}). Finally, we show that plateaued APN functions with the minimum image set size can have only a few possible value distributions (Theorem~\ref{prop: 2nd impossible}). \\

\noindent\textbf{The structure of the paper.} The paper is organized in the following way. In Section~\ref{sec: 2}, we give some preliminary results and introduce the necessary notation and definitions. In Section~\ref{sec:walsh}, we will use the Walsh transform to get information on image and preimage sets of functions $F\colon\F_p^n\to\F_p^m$. Because of the low number of distinct Walsh coefficients a plateaued function can attain, these tools will be especially useful and strong in this setting. In Section~\ref{sec 3}, we consider in detail the connection between value distributions, the Walsh transform and imbalance, a measure of distance from being balanced. Using the results of this section, we investigate value distributions of plateaued functions in Section~\ref{sec:plateaeud}. Specifically, in Subsection~\ref{sec 4.1} we investigate plateaued functions that are $d$-to-$1$ and get a connection to linearity and differential uniformity. Moreover, we completely determine the combinatorial structure of such functions in terms of their components noting that some technical details are given in Appendix~\ref{sec: Appendix} for the reader's convenience. In Subsection~\ref{sec 4.2}, we show that some known constructions of plateaued functions induce almost balanced functions. In Section~\ref{sec:APN}, we investigate the combinatorial structure of plateaued APN functions by determining possible component functions. In Section~\ref{sec: concl}, we conclude the paper and provide a list of further open problems. 

\section{Preliminaries and notation}\label{sec: 2}

Let $\F_p$ be the finite field with $p$ elements and $\F_p^n$ be the vector space of dimension $n$ over $\F_p$. For $x=(x_1,\ldots,x_n),y=(y_1,\ldots,y_n)\in\F_p^n$, we define the scalar product of $\F_p^n$ by $\langle x,y\rangle=x_1y_1+\cdots+x_ny_n$. If necessary, we endow the vector space $\F_p^n$ with the structure of the finite field $\F_{p^{n}}$; in this case, we define the scalar product of $\F_{p^n}$ by $\langle x,y\rangle =\Tr(xy)$, where $\Tr(z):=\Tr_1^n(z)$ is the absolute trace and  $\Tr^n_m(z)=\sum_{i=0}^{\frac{n}{m}-1} z^{p^{i\cdot m}}$ is the relative trace of $z \in \mathbb{F}_{p^n}$ from $\mathbb{F}_{p^n}$ into the subfield $\mathbb{F}_{p^m}$. 
If $n=2k$ is even, the vector space $\F_p^n$ can be identified with $\mathbb{F}_{p^k} \times \mathbb{F}_{p^k}$; in this case, we define the scalar product $\left\langle\left(u_1, u_2\right),\left(v_1, v_2\right)\right\rangle_n=\operatorname{Tr}_1^k\left(u_1 v_1+u_2 v_2\right)$.

For an odd prime $p$, the mappings $F\colon\F_p^n\rightarrow\F_p$ are called \textit{$p$-ary functions}, and for $p=2$, \textit{Boolean functions}. For $m\geq 2$, the mappings $F\colon\F_p^n\rightarrow\F_p^m$ are called \textit{vectorial functions}. For $b\in\F_p^m$, the function $F_b(x):=\langle b,F(x)\rangle$ is called a \textit{component function} of $F$. 

Identifying $\F_p^n$ with $\F_{p^n}$, we can uniquely represent any function $F\colon\F_p^n\to\F_p^n$  as a polynomial $F\colon\F_{p^n}\to\F_{p^n}$ of the form $F(x)=\sum_{i=0}^{p^{n}-1} a_{i} x^{i}$ with coefficients $a_i\in\F_{p^{n}}$. When $m|n$, any function $F\colon\F_p^n\to\F_p^m$ can be written as a polynomial $F\colon\F_{p^n}\to\F_{p^m}$ given by $F(x)=\Tr^{n}_{m}\left(\sum_{i=0}^{p^{n}-1} a_{i} x^{i}\right)$. This representation is called the \textit{univariate (trace) representation}, however, it is not unique in general. 

For a function $F\colon \F_p^n\to \F_p^m $, we call $F^{-1}(\beta)$ the \textit{preimage set} of $\beta \in \F_p^m$. We refer to the sizes of the preimages sets as the \textit{value distribution} of $F$. Finding the value distribution is thus equivalent to calculating the sizes of all preimages. A function $F\colon \F_p^n\to \F_p^m $ is called \textit{balanced} if $|F^{-1}(\beta)|=p^{n-m}$, for all $\beta \in \F_p^m$.

A mapping $A\colon\F_p^n\to\F_p^m$ is called \textit{affine} if it can be written as {$A(x)=L(x)+c$ for some $c \in \F_p^m$} and an $\F_p$-linear function $L \colon \F_p^n \rightarrow \F_p^m$. This means that for any $a\in\F_p^m$ and for any $b\in\F_p^m$ the equation $A(x+a)-A(x)=b$ has either $0$ or $p^n$ solutions $x\in\F_p^n$. In the following, we will be interested in functions that are opposite to affine in the sense that the equation $F(x+a)-F(x)=b$ has for fixed non-zero $a$ and any fixed $b$ only few solutions.

We define the \emph{differential uniformity }$\delta$ of a function $F\colon  \F_p^n \rightarrow \F_p^m$ as the maximum number of solutions of the equation $F(x+a)-F(x)=b$ where $a$ and $b$ range over $\F_p^n\setminus \{0\}$ and $\F_p^m$, respectively. $F$ is then said to be \textit{perfect nonlinear} (or simply \textit{bent}) if $|\{x \in \F_p^n\colon  F(x+a)-F(x)=b\}|=p^{n-m}$ holds for all $a \in \F_p^n \setminus\{0\}$  and $b \in \F_p^m$. The case $p=2$ can be considered as a special one since bent functions $F\colon\F_2^n\to\F_2^m$ exist if and only if $n$ is even and $m\le n/2$, see~\cite{KaiNy91}. Moreover, the number of solutions of the equation $F(x+a)+F(x)=b$ is always even since the solutions $x$ and $x+a$  from $\F_2^n$ are always coming in pairs. In the case $n=m$ (especially attractive from a cryptographic perspective), the definition of (almost) perfect nonlinearity reflects naturally the minimum possible number of such pairs. A function $F\colon  \F_2^n \rightarrow \F_2^n$ is said to be \textit{almost perfect nonlinear} (\textit{APN} for short) if $|\{x \in \F_2^n\colon  F(x+a)+F(x)=b\}|\le 2$ holds for all $a \in \F_2^n \setminus\{0\}$  and $b \in \F_2^n$.

Our primary tool for handling the mappings $F\colon\F_p^n\to\F_p^m$ and their value distributions is the Walsh transform, which is defined as follows. For a $p$-ary function $f\colon\F_p^n\to\F_p$, the Walsh transform is the complex-valued function $W_f\colon\F_{p}^n\to \mathbb{C}$ defined by
$$
W_f(a)=\sum_{x \in \F_p^n} \zeta_p^{f(x) -\langle a, x\rangle} \in \Z[\zeta_p], \quad\mbox{where } \zeta_p=e^{2 \pi i / p} \quad\mbox{and }i^2=-1.
$$
As it can be readily seen, the Walsh transform for $a=0$ is directly connected to the image set of $f$ and the sizes of preimages. Thus, the Walsh transform at $0$ is one of the major tools we will use in this paper.

For vectorial functions $F\colon\F_p^n\to\F_p^m$, the Walsh transform is defined using the notion of component functions as $W_F(b,a)=W_{F_b}(a)$ for all $a\in\F_p^n,b\in\F_p^m$. One can completely describe (almost) perfect nonlinearity by means of the Walsh transform as follows. A function $F\colon\F_2^n\to\F_2^n$ is APN if and only if the fourth power moment of the Walsh transform achieves the bound $\sum_{a \in \F_2^n, b \in \F_2^n \backslash\{0\}}\left(W_F(b,a)\right)^4 \geq\left(2^n-1\right) 2^{3 n+1}$ with equality, see~\cite[Corollary 1]{berger2006almost}. A function $f\colon\F_p^n\to \F_p$ is a bent function if and only if the Walsh transform satisfies $|W_f(b)|=p^{n / 2}$ for all $b \in \F_p^n$. For $m\ge 2$, a function $F\colon\F_p^n\to \F_p^m$ is vectorial bent if and only if for all $b \in \F_p^m\setminus\{0\}$ the component function $F_b\colon\F_p^n\to \F_p$ is bent, see~\cite{MeierStaffelbach89,KaiNy91}.

More generally, the \emph{linearity} of a function $F\colon\F_p^n\to\F_p^m$ is defined as $\mathcal{L}(F)=\max\limits_{a \in \F_p^n, b \in \F_p^m\setminus\{0\}}|W_F(b,a)|$ and describes how far a function is from being  {affine}. As such, the linearity is a major property of building blocks in symmetric cryptography.

The class of plateaued functions was introduced by Zheng and Zhang~\cite{ZhengZhang1999} for the case $p=2$, $m=1$. Later it was generalized for the $p$-ary case~\cite{AycyaWilfried2013} and also studied in the vectorial setting~\cite{carlet2015boolean,Mesnager2018}. The class of plateaued functions includes bent and affine functions (as extremal cases), all quadratic functions (i.e., mappings $F\colon\F_p^n\to\F_p^m$ with the property $x\mapsto F(x+a)-F(x)$ is affine for all $a\in\F_p^n$) as well as most of the constructions and the known examples of APN functions. It is defined in the following way.

\begin{definition}
    A function $f\colon \mathbb{F}_p^n \rightarrow \mathbb{F}_p$ is called $t$-plateaued, if for all $a \in \mathbb{F}_p^n$  it holds that $\left|W_f(a)\right| \in\left\{0,p^{\frac{n+t}{2}}\right\}$. The value $p^{\frac{n+t}{2}}$ is called the amplitude of a $t$-plateaued function $f$. 1-plateaued functions on $\mathbb{F}_p^n$ with $n$ odd and 2-plateaued functions on $\mathbb{F}_2^n$ with $n$ even are called semi-bent. A vectorial function $F\colon \mathbb{F}_p^n \rightarrow \mathbb{F}_p^m$ is said to be $t$-plateaued if all its component functions $F_b$ with $b \neq 0$ are $t$-plateaued. In this case we also say $F$ is plateaued with single amplitude. If all component functions $F_b$ of a function $F\colon\F_p^n\to\F_p^m$ are $s_b$-plateaued (not necessarily with the same amplitude), then $F$ is called a plateaued function.
\end{definition}
Clearly, knowing all amplitudes of a plateaued function immediately determines the linearity of the function.

In a recent article~\cite{Koelsch2023}, the authors developed the theory of value distributions of perfect nonlinear functions $F\colon\F_p^n\to\F_p^m$ relying mostly on differential properties of the considered functions. However, such an approach can not be used directly when little information about the differential properties of the considered functions is given. In the following section, we provide an alternative approach to value distributions of mappings $F\colon\F_p^n\to\F_p^m$ which is based solely on the knowledge of the Walsh transform at $a=0$.

\section{Bounds on the cardinality of preimage sets in terms of the Walsh transform at $a=0$} \label{sec:walsh}

There are two main tools we use to connect the Walsh transform to (pre)image sets. Let $F \colon \F_p^n \rightarrow \F_p^m$. Firstly, we have that, by orthogonality of characters, see e.g.~\cite[Theorem 3.3.]{hou2018lectures}
\begin{equation} \label{eq:first}
    \sum_{b \in \F_{p}^m}W_F(b,0) = \sum_{x \in \F_{p}^n}\sum_{b \in \F_{p}^m} \zeta_p^{\langle b,F(x)\rangle} = p^m |F^{-1}(0)|.
\end{equation}

Since most properties of functions we are interested in (like plateauedness, APN-ness, etc.) are invariant under the addition of constants, this information does not just apply to preimages of $0$ but can be transferred to any preimage set. Indeed, if $a \in \F_p^m$, we have for the function $G(x)=F(x)-a$
\begin{equation*} 
    \sum_{b \in \F_{p}^m}W_{G}(b,0) = \sum_{x \in \F_{p}^n}\sum_{b \in \F_{p}^m} \zeta_p^{\langle b,F(x)-a\rangle} =p^m |F^{-1}(a)|.
\end{equation*}

Secondly, the size of the set $|\{(x,y) \in \F_p^n\times \F_p^n \colon F(x)=F(y)\}|$ that is clearly intimately connected to (pre)image set sizes can be modelled by the second moment of the Walsh transform  {at $a=0$}. Indeed:
\begin{align}
       |\{(x,y) \in \F_p^n\times \F_p^n \colon F(x)=F(y)\}|&=|\{(x,y) \in \F_p^n\times \F_p^n \colon F(x)-F(y)=0\}|  \nonumber \\
       &=\frac{1}{p^m}\sum_{x,y \in \F_p^n} \sum_{b \in \F_p^m} \zeta_p^{\langle b,F(x)-F(y) \rangle} \nonumber \\
       &= \frac{1}{p^m} \sum_{b\in \F_p^m} \left(\sum_{x\in \F_p^n} \zeta_p^{\langle b,F(x) \rangle}\cdot \sum_{y\in \F_p^n} \zeta_p^{\langle b,-F(y) \rangle}\right)\nonumber \\
       &=\frac{1}{p^m} \sum_{b \in \F_p^m}W_F(b,0)\cdot \overline{W_F(b,0)} =\frac{1}{p^m} \sum_{b \in \F_p^m} |W_F(b,0)|^2 \nonumber \\
       &=p^{2n-m}+\frac{1}{p^m} \sum_{b \in \F_p^m\setminus\{0\}} |W_F(b,0)|^2, \label{eq:secondmoment}
    \end{align}
     {where we use orthogonality of characters in the second step.}
Of course, investigating the set  $|\{(x,y) \in \F_p^n\times \F_p^n \colon F(x)=F(y)\}|$ is natural whenever image sets of functions are of interest, and many papers have used this set (and in some cases its connection to the Walsh transform) in the past decades, see e.g.~\cite{uchiyama1954nombre,voloch1989number,carlet2007nonlinearities,carlet2011relating,KoelschPolujan2024,Koelsch2023,Coulter2014} for a non-exhaustive list. Noting that 
    \begin{equation} \label{eq:walsh}
        \sum_{\beta \in \F_{p}^m} |F^{-1}(\beta)|^2=|\{(x,y) \in \F_p^n\times \F_p^n \colon F(x)=F(y)\}|,
    \end{equation}
    this immediately yields a result connecting preimage set sizes to the Walsh transform of a function  {at $a=0$}:

\begin{proposition} \label{prop:starting}
Let $F \colon \F_{p}^n \rightarrow \F_{p}^m$ be a function. Then 
\[\sum_{\beta \in \F_{p}^m} |F^{-1}(\beta)|^2 = p^{2n-m}+\frac{1}{p^m} \sum_{b \in \F_p^m\setminus\{0\}} |W_F(b,0)|^2.\]
\end{proposition}
\begin{proof}
    Follows immediately from a combination of Eqs.~\eqref{eq:secondmoment} and~\eqref{eq:walsh}.
    \end{proof}

\begin{remark}
    Proposition~\ref{prop:starting} is a generalization of known results with similar statements, for instance for bent functions, we have $|W_F(b,0)|^2=p^n$ for all $b \in \F_p^m\setminus\{0\}$ and we recover the result from~\cite[Proposition 2.1.]{KoelschPolujan2024} that was proven using different means.
\end{remark}

The sum of the second moment of the Walsh transform  {at $a=0$},  $\frac{1}{p^m}\sum_{b \in \F_p^m\setminus\{0\}} |W_F(b,0)|^2$ that appears in Proposition~\ref{prop:starting} will play a crucial role throughout this paper, we thus introduce the notation
\[\Nb_F:=\frac{1}{p^m}\sum_{b \in \F_p^m\setminus\{0\}} |W_F(b,0)|^2.\]
This notation was already used by Carlet and Ding in~\cite{carlet2007nonlinearities} for vectorial Boolean functions, and is called the \emph{imbalance} of $F$.
Note that $\Nb_F$ has to be an integer by Proposition~\ref{prop:starting} as long as $m \leq 2n$. Immediate bounds on $\Nb_F$ are (see~\cite[Proposition 1]{carlet2007nonlinearities}, they can also be derived from the definition above with ease) $0 \leq \Nb_F \leq p^{2n}-p^{2n-m}$. The lower bound is achieved with equality if and only if $F$ is balanced, and the upper bound if and only if $F$ is constant.

For the sake of brevity, denote for a function $F \colon \F_p^n \rightarrow \F_p^m$ the preimage set sizes by $X_1,X_2,\dots,X_{p^m}$, where we order the $X_i$ without loss of generality such that $0<X_1\leq X_2 \leq \dots X_{|\image(F)|}$ and $X_i=0$ for all $i>|\image(F)|$.

\begin{theorem} \label{th:cs_general}
    Let $F \colon \F_{p}^n \rightarrow \F_{p}^m$ be a function. For each $\beta \in \F_p^m$ with $|F^{-1}(\beta)|>0$ we have
    \begin{equation}\label{eq: bounds}
        \frac{p^n}{|\image(F)|}-\Xi(F)\le |F^{-1}(\beta)| \le \frac{p^n}{|\image(F)|}+\Xi(F),\quad \mbox{where}
    \end{equation}
    \begin{equation}\label{eq: defect}
       \Xi(F)=\sqrt{\frac{(|\image(F)|-1)  \left( |\image(F)| \Nb_F -p^{2 n-m} \left(p^m-|\image(F)|\right)\right)}{|\image(F)|^2}}.
    \end{equation}
    Moreover, if $\beta \in \F_p^m$ satisfies one of the inequalities above with equality then $|F^{-1}(\alpha)|=\frac{p^n-|F^{-1}(\beta)|}{|\image(F)|-1}$ for all $\alpha \in \image(F)\setminus \{\beta\}$.
\end{theorem}
\begin{proof}
   It is easy to observe that the theorem holds if $|\image(F)|\leq 1$, we can thus assume $|\image(F)|\geq 2$ for the rest of the proof.
   
   Note that $\sum_{i=1}^{|\image(F)|}X_i=p^n$ since the set of preimages partitions $\F_p^n$.
   By the Cauchy-Schwarz inequality, we  then have for any $1 \leq j \leq |\image(F)|$
   \[\sum_{\substack{i=1\\i\neq j}}^{|\image(F)|} X_i^2 \geq  \left(\sum_{\substack{i=1\\i\neq j}}^{|\image(F)|} X_i\right)^2 \cdot \frac{1}{|\image(F)|-1} = (p^n-X_j)^2\cdot \frac{1}{|\image(F)|-1},\]
		with equality if and only if all $X_i$, $1\leq i\leq |\image(F)|$, $i\neq j$ are identical. Then, applying Proposition~\ref{prop:starting},
		\[p^{2n-m} + \Nb_F=\sum_{i=1}^{|\image(F)|} X_i^2 =X_j^2+\sum_{\substack{i=1\\i\neq j}}^{|\image(F)|} X_i^2 \geq X_j^2+\frac{(p^n-X_j)^2}{|\image(F)|-1}.\]
		This inequality is quadratic in $X_j$ and can be solved with elementary techniques, the result is
  $$\frac{p^n}{|\image(F)|}-\Xi(F)\le X_j \le \frac{p^n}{|\image(F)|}+\Xi(F),\quad \mbox{where}$$
  $$\Xi(F)=\sqrt{\frac{(|\image(F)|-1)  \left( |\image(F)| \Nb_F -p^{2 n-m} \left(p^m-|\image(F)|\right)\right)}{|\image(F)|^2}},$$
  proving the first part of the Theorem. If $\beta \in  \F_p^m$ satisfies Eq.~\eqref{eq: bounds} with equality, then equality has to hold also in the Cauchy-Schwarz inequality.
  Consequently, for all $\alpha \in \image(F)\setminus\{\beta\}$, we have that  $|F^{-1}(\alpha)|=\frac{p^n-|F^{-1}(\beta)|}{|\image(F)|-1}$ as claimed.
\end{proof}

\begin{remark}
    For a function $F \colon \F_{p}^n \rightarrow \F_{p}^m$, the value $\Xi(F)$ defined by Eq.~\eqref{eq: defect} can be considered as an imbalance defect of $F$, since $F$ is not balanced if $\Xi(F)\neq0$. We also note that the imbalance defect is always a real number. Indeed, the value under the square root is negative if and only if $ |\image(F)| \Nb_F -p^{2 n-m} \left(p^m-|\image(F)|\right)<0$, which (using Proposition~\ref{prop:starting}) is equivalent to  
    \[\sum_{\beta \in \F_p^m}|F^{-1}(\beta)|^2-\frac{p^{2n}}{|\image(F)|}<0.\]
    But elementary calculations show that 
    \[\sum_{\beta \in \F_p^m}|F^{-1}(\beta)|^2-\frac{p^{2n}}{|\image(F)|} =  {\sum_{\beta \in \image(F)}|F^{-1}(\beta)|^2-\frac{p^{2n}}{|\image(F)|} =}\sum_{\beta \in  {\image(F)}}\left(|F^{-1}(\beta)|-\frac{p^n}{|\image(F)|}\right)^2 \geq 0,\]
    giving the desired contradiction.
\end{remark}

If no further information on the size of the image set of $F$ is known, bounds on the preimage set sizes can still be derived just from $\Nb_F$ by following the proof of Theorem~\ref{th:cs_general}, substituting $|\image(F)|$ with $p^m$ throughout. The result is:
\begin{proposition} \label{prop:cs}
    Let $F \colon \F_{p}^n \rightarrow \F_{p}^m$ be a function. For each $\beta \in \F_p^m$ we have
    \[p^{n-m}-\sqrt{\left(1-\frac{1}{p^m}\right)\Nb_F} \leq |F^{-1}(\beta)| \leq p^{n-m}+\sqrt{\left(1-\frac{1}{p^m}\right)\Nb_F}.\]
    Moreover, if a $\beta \in \F_p^m$ satisfies one of the inequalities above with equality then $|F^{-1}(\alpha)|=\frac{p^n-|F^{-1}(\beta)|}{p^m-1}$ for all $\alpha \in \F_p^m\setminus \{\beta\}$.
\end{proposition}

\begin{remark}
    Clearly, for balanced functions $\Nb_F=0$. For bent functions $\Nb_F=\frac{1}{p^m}(p^m-1)p^n=p^n-p^{n-m}$ and $\sqrt{(1-p^{-m})\Nb_F} = \sqrt{p^n-p^{n-m}-p^{n-m}+p^{n-2m}}=\sqrt{(p^{n/2}-p^{n/2-m})^2}=p^{n/2}-p^{n/2-m}$, recovering the result from~\cite[Theorem 2.5.]{KoelschPolujan2024}.
\end{remark}

It is well-known that a function $F \colon \F_{p}^n \rightarrow \F_{p}^m$  is balanced if and only if $W_F(b,0)=0$ for all  $b \in \F_{p}^m \setminus \{0\}$. Proposition~\ref{prop:cs} gives a quantitative version of this well-known result: It gives a bound that indicates how far away a function can be from being balanced, based on the value of $\Nb_F$. Note that the bounds in Proposition~\ref{prop:cs} are in general tight, vectorial bent functions having preimages satisfying the bounds  above with equality were investigated in~\cite{KoelschPolujan2024} in detail. Following the original paper~\cite{KoelschPolujan2024}, we introduce the notion of almost balanced functions as follows.

\begin{definition}\label{def: AB}
    Let  $F \colon \F_{p}^n \rightarrow \F_{p}^m$ be a function with $\Xi(F)\neq 0$. Then, $F$ is called \textit{almost balanced} if there is a $\beta \in \F_p^m$ that achieves one of the bounds in~\eqref{eq: bounds} with equality. If the lower bound is achieved, $F$ is called almost balanced of type $(-)$. If the upper bound is achieved, $F$ is called almost balanced of type $(+)$.
\end{definition}
The motivation behind the term ``almost balanced'' is that all elements in the image set of these functions except for one have the same amount of preimages. These functions are in this sense the ``closest'' to being balanced.
 \begin{remark}
    We want to note that if the image set size of a function is very small, the almost balanced property is not particularly interesting. For instance, all functions with only $2$ elements in the image set are almost balanced by default. All explicit families of almost balanced functions we study later have large image sets and thus do not fall into these trivial cases.
\end{remark}
In the following section, we study almost balanced functions in more detail.

\section{Imbalance, value distributions, and the Walsh transform}\label{sec 3}

It is an obvious question for which values of $\Nb_F$ the inequalities in Proposition~\ref{prop:cs} can hold with equality. It turns out that strong conditions occur and we can make substantial statements on the Walsh transform  {at $a=0$}, and thus the imbalance $\Nb_F$ of the function.

\begin{theorem} \label{thm:AB}
       Let $F \colon \F_p^n \rightarrow \F_p^m$ be a surjective, almost balanced function where $F^{-1}(0)$ is the preimage with unique preimage size. Then 
    \begin{enumerate}[(a)]
        \item  $W_F(b,0)=W_F(\mathbf{e}_1,0)$ for all $b \in \F_p^m\setminus\{0\}$, where $\mathbf{e}_1=(1,0\dots,0)\in \F_p^m$,
        \item $W_F(b,0) \in \Z \setminus\{0\}$  for all $b \in \F_p^m$,
        \item $\Nb_F=\frac{p^m-1}{p^m}|W_F(\mathbf{e}_1,0)|^2$,
        \item $p^m|W_F(\mathbf{e}_1,0)$,
    \end{enumerate} 
    and
    \[|F^{-1}(0)|=p^{n-m} \pm \frac{p^m-1}{p^m}|W_F(\mathbf{e}_1,0)| \text{ and }|F^{-1}(\beta)|=p^{n-m} \mp \frac{1}{p^m}|W_F(\mathbf{e}_1,0)|\]
    for any $\beta \in \F_p^m\setminus \{0\}.$
\end{theorem}
\begin{proof}
    Combining Eq.~\eqref{eq:first} with Proposition~\ref{prop:cs} yields
    \[|F^{-1}(0)|=\frac{1}{p^m}\sum_{b \in \F_p^m}W_F(b,0) = p^{n-m} \pm \sqrt{\left(\frac{1}{p^m}-\frac{1}{p^{2m}}\right)\sum_{b \in \F_{p^m}\setminus\{0\}}|W_F(b,0)|^2}.\]
    Noting $W_F(0,0)=p^n$, this implies
    \[\mp \sum_{b \in \F_{p^m}\setminus\{0\}}W_F(b,0)=\sqrt{(p^m-1)\sum_{b \in \F_{p^m}\setminus\{0\}}|W_F(b,0)|^2} \geq \sum_{b \in \F_{p^m}\setminus\{0\}}|W_F(b,0)|,\]
    where we use the Cauchy-Schwarz inequality in the last step. Clearly, this inequality can only hold with equality, so all $|W_F(b,0)|$ for $b\in \F_p^m\setminus \{0\}$ are the same by the Cauchy-Schwarz inequality. In fact, all $W_F(b,0)$ have to be the same to satisfy the inequality. Then $W_F(b,0) \in \Z$ (since the left-hand side of the inequality has to be an integer), and, clearly, $W_F(b,0)\neq 0$ since otherwise $F$ would be balanced. This proves (a) and (b). Item (c) follows then immediately from the definition of $\Nb_F$, and (d) follows from (c). The statement on the preimage sizes is then just a restatement of Proposition~\ref{prop:cs}.
    \end{proof}

\begin{remark}
    If the preimage with unique preimage set size is $F^{-1}(a)$ with $a \neq 0$, part (c) of Theorem~\ref{thm:AB} (and thus everything relating to $\Nb_F$) still holds. Indeed, since a shift by a constant multiplies the values of $W_F(b,0)$ only by a root of unity, the statement that $|W_F(b,0)|$ is constant for all nonzero $b$ remains true, and thus $\Nb_F$ is unchanged.
\end{remark}

Theorem~\ref{thm:AB}, especially (a) and (d), leads us in a natural way to consider almost balanced plateaued functions: Indeed, since plateaued functions always satisfy that $W_F(b,0) \in \{0,\pm p^{(n+k)/2}\}$, they are natural candidates for almost balanced functions. We can immediately deduce from Theorem~\ref{thm:AB}:
\begin{corollary}
    Let $F\colon \F_p^n \rightarrow \F_p^m$ be an almost balanced surjective plateaued function. Then $F$ is plateaued with single amplitude $p^{(n+t)/2}$. Moreover, we have that $n+t$ is even and $m\leq (n+t)/2$. 
\end{corollary}
\begin{proof}
    $F$ is plateaued with single amplitude $p^k$ by Theorem~\ref{thm:AB} (a). We can shift $F$ such that the resulting function $G$ has $G^{-1}(0)$ as the preimage with unique preimage size, $G$ will clearly still be plateaued with single amplitude $p^{(n+t)/2}$. We then have $W_G(b,0)\in\Z$ by Theorem~\ref{thm:AB} (b), so $p^{(n+t)/2}$ has to be an integer, implying that $n+t$ is even. The last statement $m\leq (n+t)/2$ then follows from Theorem~\ref{thm:AB} (d).
\end{proof}

Of course, if $F$ is surjective, then the bounds of Theorem~\ref{th:cs_general} and Proposition~\ref{prop:cs} coincide. We now give a condition, only based on $\Nb_F$ when a function is necessarily surjective:
\begin{theorem} \label{thm:surj}
    Let $F \colon \F_{p}^n \rightarrow \F_{p}^m$ be a function. If {
    \[\Nb_F <\frac{p^{2n-m}}{p^m-1}\]
    }then $F$ is surjective.
\end{theorem}
\begin{proof}
    By Proposition~\ref{prop:cs}, we have for all $\beta \in \F_{2^m}$
    \[ |F^{-1}(\beta)| \geq p^{n-m}-\sqrt{\left(1-\frac{1}{p^m}\right)\Nb_F}.\]
    So if the right-hand side of this equation is positive, $F$ is necessarily surjective. This leads immediately to the result.
\end{proof}

The following is a straightforward generalization of~\cite[Proposition 3]{carlet2011relating}, where the result was proven only for $p=2$; the proof in~\cite{carlet2011relating} transfers seamlessly to $p>2$.

\begin{lemma} \label{lem:carlet}
    Let $F \colon \F_{p}^n \rightarrow \F_{p}^m$ be a function. Then there exists an $\F_p$-linear function $L \colon \F_{p}^n \rightarrow \F_{p}^m$ such that $\Nb_{F+L}\leq p^n-p^{n-m}$.
\end{lemma}

We can use this to show that if the output dimension of a vectorial function is small enough, \emph{every} function is equivalent to a surjective function in the sense that there is an additive mapping $L$ such that $F+L$ is surjective.

\begin{theorem}\label{th: everything is surjective}
    Let $F \colon \F_{p}^n \rightarrow \F_{p}^m$ be a function with $m\leq n/2$. Then there exists an $\F_p$-linear function $L\colon \F_{p}^n \rightarrow \F_{p}^m$ such that $F+L$ is surjective.
\end{theorem}
\begin{proof}
    By Lemma~\ref{lem:carlet} we have for some $L$ that $\Nb_{F+L} \leq p^n-p^{n-m}$. If $m \leq n/2$ we have $$\Nb_{F+L}\leq p^n-p^{n-m} <p^{2(n-m)}<\frac{p^{2(n-m)}}{1-\frac{1}{p^m}},$$
    and the result follows from Theorem~\ref{thm:surj}.
\end{proof}

The following result on the image set size of an arbitrary function based on $\Nb_F$ is from~\cite[Lemma 1]{Koelsch2023}, see also~\cite[Lemma 5]{carlet2017trade}; we include it for completeness and give a slightly different proof based on our previous results in this paper.
\begin{proposition} \label{prop:imageset}
    Let $F \colon \F_{p}^n \rightarrow \F_{p}^m$ be a function. Then
    \[|\image (F)| \geq \frac{p^{2n}}{p^{2n-m}+\Nb_F}.\]
\end{proposition}
\begin{proof}
    We have from Proposition~\ref{prop:starting}
		\begin{align*}
	\sum_{i=1}^{|\image(F)|} X_i^2&=p^{2n-m}+\Nb_F \\
	\sum_{i=1}^{|\image(F)|} X_i&=p^n,
\end{align*}
	and again by the Cauchy-Schwarz inequality, the following holds
	\begin{equation}\label{eq: Images set cardinality bound aux}
	    p^{2n-m}+\Nb_F=\sum_{i=1}^{|\image(F)|} X_i^2\geq\frac{p^{2n}}{|\image(F)|}.
	\end{equation}
	The claim follows by solving Eq.~\eqref{eq: Images set cardinality bound aux} for $|\image(F)|$.
\end{proof}

\section{Applications to plateaued functions} \label{sec:plateaeud}

In the previous sections, many connections between $\Nb_F$ and (pre)image set sizes were drawn. Clearly, information on the Walsh coefficients of a function can thus be turned into information on these (pre)image sets. A natural candidate are plateaued functions, only allowing $W_F(b,0)\in \{0,\pm p^{\frac{n+k}{2}}\}.$ This connection was already used for special cases of plateaued functions: when $F$ is bent~\cite{KoelschPolujan2024} or APN with classical Walsh transform~\cite{Koelsch2023}. Here, we give a unified and complete treatment, encompassing all these results as well as giving some new results. Additionally, we show that some known constructions of plateaued functions induce almost balanced functions.
\subsection{Plateaued functions that are $d$-to-$1$}\label{sec 4.1}
Recall that we call a function $F \colon \F_p^n \rightarrow \F_p^n$ a $d$-to-$1$ function if $d|p^n-1$ and $F$ has exactly one preimage set of size $1$, and all other nonempty preimage sets have size exactly $d$. The image set  of $F$ thus satisfies $|\image(F)|=\frac{p^n+d-1}{d}$. The set of $d$-to-$1$ functions contains many interesting sets of functions: Firstly, all functions defined by monomials $F(x)=x^t \in \F_{p^n}[x]$ are $\gcd(t,p^n-1)$-to-$1$. Functions defined by monomials have been of great interest, for instance, many known APN functions are monomials. On top of that, plateaued $2$-to-$1$ functions (for $p$ odd) and $3$-to-$1$ functions (for $p=2$, $n$ even) have also intimate connections to perfect nonlinear and almost perfect nonlinear functions, see~\cite{weng2012further,KoelschPolujan2024,Koelsch2023}.

Using our tools, we are able to show that $d$-to-$1$ vectorial functions can only be plateaued for very specific values of $d$. We also derive precise information on the amplitudes of the components. 

The case of odd $p$ requires a lot more attention since an additional step is needed that shows that certain Walsh coefficients are integers; it is also the only part of the paper where we make use of results on cyclotomic fields. We thus move the proof of this auxiliary lemma to Appendix~\ref{sec: Appendix}.

\begin{lemma} \label{lem:podd}
    Let $F \colon \F_p^n \rightarrow \F_p^n$ be a plateaued function that is $d$-to-1 for some $d|p^n-1$, $d>2$, where $F(0)=0$ and $|F^{-1}(0)|=1$. Then $W_F(b,0) \in \Z$ for any $b \in \F_p^n$.
\end{lemma}

\begin{theorem} \label{thm:platdto1}
    Let $F \colon \F_p^n \rightarrow \F_p^n$ be a plateaued function that is $d$-to-1 for some $d|p^n-1$, $d>2$. Then $n$ is even, $d=p^t+1$ for some $t \geq 1$, $t|n/2$ and $F$ has $p^n-1-\frac{p^n-1}{p^t+1}$ bent components and $\frac{p^n-1}{p^t+1}$ components with amplitude $p^{n/2+t}$. The linearity of such a function is thus $\mathcal{L}(F)=p^{n/2+t}$.
\end{theorem}
\begin{proof}
    If $F(x)$ is plateaued then so is $F(x+a)+b$ for any constants $a,b$ and all amplitudes of the component functions remain the same, so we can assume without loss of generality that $F^{-1}(0)=\{0\}$, i.e., $F(0)=0$ and $0$ is the unique element with exactly one preimage.

    Since $F$ is $d$-to-$1$ we have $W_F(b,0) \equiv 1 \pmod{d}$, so $W_F(b,0)-1=d\cdot C$ for some $C \in \Z[\zeta_p]$. In particular, $W_F(b,0) \neq 0$ for all $b \in \F_p^n \setminus \{0\}$. By Lemma~\ref{lem:podd}, we have that $W_F(b,0) \in \Z$, so $W_F(b,0) \in \{-p^{k},p^k\}$.
    
    This implies that $d|p^{k}\pm 1$. 
    Let us first assume that $d\nmid p^t+1$ for all $0<t<n$. Then $W_F(b,0) =  p^{k}$ for some $k$ for any nonzero $b$. But we have (see Eq.~\eqref{eq:first})
    \begin{equation} \label{eq:firstmoment}
        \sum_{b \in \F_p^n \setminus \{0\}}W_F(b,0) =\sum_{b \in \F_p^n}W_F(b,0)-p^n = p^n(|F^{-1}(0)|-1)=0,
    \end{equation}
    yielding a contradiction. We conclude that there is a minimal $0<t<n$ such that $d|p^t+1$.  Assume now that $k\geq n/2$ is another integer such that $W_F(b,0)= \pm p^{k}$ for some $b$. Then $d|p^{k} \pm 1$ and $d|p^t+1$, so in particular $d|p^{\gcd(t,k)}+1$. By the minimality of $t$, we have $\gcd(t,k)=t$. We conclude that the only possible values of $W_F(b,0)$ are $W_F(b,0)= \pm p^{k+it}$ with $i \geq 0$ for some $k \geq n/2$ that is divisible by $t$. Define 
    \[N_i = |\{b\in \F_2^n \colon W_F(b,0)=\pm p^{k+it}\}|\]
    for $i\geq 0$. {Note that for a fixed $i$, only one sign can occur since otherwise $d|p^{k+it}-1$ and $d|p^{k+it}+1$, so $d|2$, contradicting our assumption.} 
    Then we have using Eq.~\eqref{eq:firstmoment}
    \[0=\sum_{b \in \F_p^n \setminus \{0\}}W_F(b,0) = p^{k}(\pm N_0\pm p^tN_1 \pm p^{2t}N_2\pm \dots),\]
    so 
    \begin{equation} \label{eq:1}
        N_0\mp p^tN_1 \mp p^{2t}N_2\mp \dots =0.
    \end{equation}
    Moreover, we have by Proposition~\ref{prop:starting}
    \[    1+d(p^n-1)=\sum_{\beta \in \F_{p}^n} |F^{-1}(\beta)|^2 = p^n + \frac{1}{p^n}\sum_{\beta \in \F_p^n}|W_F(b,0)|^2,
    \]
    so 
    \begin{equation} \label{eq:2}
        (d-1)(p^n-1)=p^{2k-n}(N_0+p^{2t}N_1+p^{4t}N_2+\dots).
    \end{equation}
    Since we have no balanced components, we also have
    \begin{equation} \label{eq:3}
        {N_0}+N_1+N_2+\dots = p^n-1.
    \end{equation}
    We now subtract Eq.~\eqref{eq:3} $(d-1)$-times from Eq.~\eqref{eq:2} and get
    \[(p^{2k-n}+1-d)N_0+(p^{2k-n+2t}+1-d)N_1 + (p^{2k-n+4t}+1-d)N_2+\dots=0.\]
    From this equation, we add Eq.~\eqref{eq:1} $(p^t-1)$-times:
    \[(p^{2k-n}-d+p^t)N_0+(p^{2k-n+2t}+1-d\pm (p^{2t}- p^t))N_1+(p^{2k-n+4t}+1-d\pm(p^{3t}-p^{2t}))N_2+\dots =0.\]
    Note that $d\leq p^t+1$ and $p^{2k-n}\geq 1$, so all coefficients that appear in front of the $N_i$ are non-negative. More specifically, we see that all coefficients are positive, unless $n=2k$, $d=p^t+1$, in which case the coefficients of $N_0$ and $N_1$ are $0$ (when choosing the correct sign in the case of $N_1$), while the others are still positive. We conclude that necessarily $N_i=0$ for all $i \geq 2$ and $n=2k$, $d=p^t+1$. Eqs.~\eqref{eq:2} and~\eqref{eq:3} then yield $N_0+N_1=p^n-1$ and $p^t(p^n-1)=N_0+p^{2t}N_1$, leading to (after solving the two linear equations) $N_0=p^n-1-\frac{p^n-1}{p^t+1}$ and $N_1=\frac{p^n-1}{p^t+1}$. Since $t$ divides $k$ and $2k=n$, we have that $t$ necessarily divides $n/2$.
\end{proof}

\begin{remark}
\begin{enumerate}[(a)]
    \item     The case $d=2$ also occurs if $p$ is odd. The corresponding theorem was proven in~\cite[Theorem 6.4.]{KoelschPolujan2024}: A plateaued 2-to-1 function is planar, i.e., all component functions are bent. So the amplitudes given in Theorem~\ref{thm:platdto1} are still true for $2$-to-$1$ functions, setting $t=0$. Interestingly, $n$ does not have to be even in this case: Planar $2$-to-$1$ functions exist for odd $n$, indeed $F(x)=x^2$ is a simple example that works for all $n$.\\
    \item Plateaued functions of the form $F(x)=G(x^k)$ on $\F_{p^n}$ were used in~\cite[Theorem 2]{carlet2016quadratic} to construct difference sets and strongly regular graphs. Clearly, these functions are necessarily $d$-to-$1$ for $d=\gcd(k,p^n-1)$. A condition in the statement of this theorem in \cite{carlet2016quadratic} is that $d=p^t+1$ for some $t\geq 0$. Theorem~\ref{thm:platdto1} shows that this condition is in fact always satisfied (and in fact shows other necessary conditions, like that $t$ has to divide $n/2$ if $t \geq 1$).
\end{enumerate}

\end{remark}

Applying Theorem~\ref{thm:platdto1} to monomials we get:
\begin{corollary} \label{cor:monom}
    Let $F(x)=x^d$ define a plateaued mapping on $\F_{p^n}$. Then
    \begin{itemize}
        \item If $n$ is odd: $\gcd(d,p^n-1)\in \{1,2\}$, 
        \item If $n$ is even: $\gcd(d,p^n-1)\in \{1,p^t+1\}$ for some $t \geq 0$.  
    \end{itemize}
    Here the case $\gcd(d,p^n-1)=2$ can only occur for $p>2$.
\end{corollary}
    So, essentially, plateaued mappings induced by monomials are either permutations or $(p^t+1)$-to-$1$. 
    
We first investigate the latter case in Corollary~\ref{cor:monom}, i.e., when $F$ is $(p^t+1)$-to-1. It turns out that one can give precise statements on the differential uniformity of $F$ as well.
\begin{theorem} \label{thm:diff}
     Let $F \colon \F_p^n \rightarrow \F_p^n$ be a plateaued function that is $(p^t+1)$-to-1. Then $F$ has differential uniformity at least $p^t$. $F$ has differential uniformity $p^t$ if and only if $F$ is differentially at most two valued, i.e., $F(x+a)-F(x)=b$ has for any non-zero $a$ and any $b$ either $0$ or $p^t$ solutions.
\end{theorem}
\begin{proof}
    Since $F$ is $(p^t+1)$-to-1, it is at least $p^t$-uniform by~\cite[Theorem 2]{Koelsch2023}. So let us assume the differential uniformity is exactly $p^t$. 
    
    We consider the fourth moment of the Walsh transform. By Theorem~\ref{thm:platdto1} we know that $F$ has $p^n-1-\frac{p^n-1}{p^t+1}$ bent components and $\frac{p^n-1}{p^t+1}$ components with amplitude $p^{n/2+t}$. By Parseval's equation, a plateaued component function $\langle b,F(x) \rangle$ with amplitude $p^{n/2+t}$ has exactly $p^{n-2t}$ nonzero values in its Walsh transform. Then
    \begin{align*}
         \sum_{\substack{b \in \F_p^n\setminus\{0\},\\a \in \F_p^n}}|W_F(b,a)|^4 &= \left(p^n-1-\frac{p^n-1}{p^t+1}\right)p^{3n}+\frac{p^n-1}{p^t+1}\cdot p^{3n+2t}=p^{3n}\left(p^n-1+\frac{(p^n-1)(p^{2t}-1)}{p^t+1}\right)\\&=p^{3n}(p^n-1+(p^n-1)(p^t-1))=p^{3n+t}(p^n-1).
    \end{align*}
      We also have (compare e.g.~\cite[Lemma 1]{charpin2019new} for the $p=2$ case):
    \begin{align}
        \sum_{\substack{b \in \F_p^n\setminus\{0\},\\a \in \F_p^n}} |W_F(b,a)|^4 
        &=\sum_{\substack{b \in \F_p^n\setminus\{0\},\\a,x,y,z,w \in \F_p^n}}\zeta_p^{\langle b,F(x)+F(y)-F(z)-F(w) \rangle + \langle a,x+y-z-w \rangle}\nonumber\\
        &=p^n \sum_{\substack{b \in \F_p^n\setminus\{0\},\\x,y,z \in \F_p^n}} \zeta_p^{\langle b,F(x)+F(y)-F(z)-F(x+y-z) \rangle}\nonumber\\
        &=p^n \sum_{\substack{b \in \F_p^n\setminus\{0\},\\c,y,z \in \F_p^n}} \zeta_p^{\langle b,F(z+c)+F(y)-F(z)-F(c+y) \rangle}\nonumber\\
       &=p^n \sum_{\substack{b \in \F_p^n\setminus\{0\},\\c,y,z \in \F_p^n}} \zeta_p^{\langle b,(D_cF)(z)-(D_cF)(y) \rangle}\nonumber\\
        &=p^{2n}\sum_{\substack{c \in \F_p^n\setminus\{0\},\\d \in \F_p^n}}|(D_cF)^{-1}(d)|^2 \nonumber\\ 
        &\leq p^{2n}\cdot p^t\sum_{\substack{c \in \F_p^n\setminus\{0\},\\d \in \F_p^n}}  |(D_cF)^{-1}(d)|\nonumber \\
        &=p^{3n+t}(p^n-1),\label{eq:interm}
    \end{align}
        where we use the notation $(D_cF)(x)=F(x+c)-F(x)$ for the difference mapping and the fact that $\sum_{d \in \F_{p}^n}|(D_cF)^{-1}(d)|=p^n$. Note that equality holds if and only if $|(D_cF)^{-1}(d)| \in \{0,p^t\}$, so if $F$ is at most differentially $2$-valued. Comparing the two expressions of the fourth moment of the Walsh transform then yields that the inequality above has to be sharp, i.e., $|(D_aF)^{-1}(b)| \in \{0,p^t\}$.
\end{proof}

Note that the mappings defined by $F(x)=x^{p^t+1}$ give immediate examples of such plateaued functions with differential uniformity $p^t$. Mappings that are differentially two valued have connections to design theory and certain codes, we refer the reader to~\cite{tang2019codes}.

    The second case in Corollary~\ref{cor:monom}, i.e., the case of plateaued power permutations $F=x^d$ with $\gcd(d,p^n-1)=1$, also leads to strong necessary conditions. Connections between the differential uniformity of plateaued power permutations and the amplitude (or equivalently, the nonlinearity) can then be derived as well, mirroring the results of Theorem~\ref{thm:diff}. We start with a more general result, from which the result on monomials follows directly.

\begin{theorem} \label{thm:gcd1}
Let $F\colon \F_{p}^n\rightarrow \F_{p}^n$ be a $t$-plateaued function for some $t>0$.
    Then its differential uniformity $\delta$ is at least $p^t$. $F$ has differential uniformity $p^t$ if and only if $F$ is differentially two valued, i.e., $F(x+a)-F(x)=b$ has for any non-zero $a$ and any $b$ either $0$ or $p^t$ solutions.
\end{theorem}
\begin{proof}
    We have (see Eq.~\eqref{eq:interm}, with $\delta$ replacing $p^t$):
        \[\sum_{\substack{b \in \F_p^n\setminus\{0\}\\a \in \F_p^n}} |W_F(b,a)|^4 \leq p^{3n}(p^n-1)\delta,\]
    where equality holds if and only if $|(D_aF)^{-1}(b)| \in \{0,\delta\}$, so if $F$ is differentially $2$-valued.

    $F$ is $t$-plateaued for some $t$, i.e., $|W_F(b,a)|^2 \in \{0,p^{n+t}\}$ for $b\neq 0$. Parseval's equation states $\sum_{a \in \F_p^n}|W_F(b,a)|^2=p^{2n}$ for any $b \in \F_{p}^n$, so a plateaued component function $\langle b,F(x) \rangle$ with amplitude $p^{(n+t)/2}$ has exactly $p^{n-t}$ nonzero values in its Walsh transform, i.e., $W_F(b,a) \neq 0$ for exactly $p^{n-t}$ values of $a$. 
    So 
    \[\sum_{\substack{b \in \F_p^n\setminus\{0\},\\a \in \F_p^n}}  |W_F(b,a)|^4 = (p^n-1)\cdot p^{n-t}\cdot p^{2n+2t}=(p^n-1)p^{3n+t}.\]
    Combining the two statements on the fourth moment of the Walsh transform, we get $\delta \geq p^{t}$, with equality if and only if $F$ is differentially two-valued. 
\end{proof}

\begin{corollary}
        Let $F(x)=x^d$ define a plateaued mapping on $\F_{p}^n$ with $\gcd(d,p^n-1)=1$. Then $F$ is $t$-plateaued for some $t>0$ and its differential uniformity $\delta$ is at least $p^t$. $F$ has differential uniformity $p^t$ if and only if $F$ is differentially two valued.
\end{corollary}
\begin{proof}
        It is easy to see and well known that in this case $W_F(b,a)=W_F(1,a/b^{1/d})$ where $1/d$ is the inverse of $d$ modulo $p^n-1$. This immediately implies that $F$ is $t$-plateaued for some $t>0$, and the statement follows from Theorem~\ref{thm:gcd1}.
\end{proof}

\subsection{Almost balanced plateaued functions}\label{sec 4.2}
In~\cite{KoelschPolujan2024}, it was shown that many known theoretical constructions of vectorial bent functions are almost balanced. Now, we show that similar constructions of plateaued functions also naturally induce almost balanced ones. We begin with the analysis of monomial functions.
\begin{theorem} \label{thm:gold}
    Let $n/\gcd(r,n)$ be even and $F \colon \F_{2^n} \rightarrow \F_{2^{n/2}}$ be defined by $F(x)=\Tr^n_{n/2}(x^{2^r+1})$. Then $F$ is plateaued with single amplitude $2^{n/2+d}$ where $d=\gcd(r,n)$. Furthermore, if $d\neq n/2$, then  $F$ is surjective and almost balanced with preimage set sizes
    \[|F^{-1}(0)|=2^{n/2}+(2^{n/2}-1)\cdot 2^d \text{ and }|F^{-1}(c)|=2^{n/2}-2^d\]
    for any $c \in \F_{2^{n/2}}^*$.
\end{theorem}
\begin{proof}
    Consider the plateaued function $G \colon \F_{2^n} \rightarrow \F_{2^n}$ defined by $G(x)=x^{2^r+1}$. Then the component functions $G_\lambda(x) = \Tr^n_1(\lambda x^{2^r+1})$ are bent if and only if $\lambda$ is not a $(2^r+1)$-st power (see~\cite{gold1968maximal}, or also~\cite{dong2013note}). Further, $G$ is $(2^d+1)$-to-$1$ since $\gcd(2^r+1,2^n-1)=2^d+1$. So, by Theorem~\ref{thm:platdto1}, the other component functions have to be plateaued with amplitude $2^{n/2+d}$. So $G_\lambda = \Tr^n_1(\lambda x^{2^r+1})$ is plateaued with amplitude $2^{n/2+d}$ if and only if $\lambda$ is a $(2^d+1)$-st power. 

    Now consider the component functions of $F$. We have for $\lambda \in \F_{2^{n/2}}$ that $$F_\lambda(x)=\Tr^{n/2}_1(\lambda \Tr^n_{n/2}(x^{2^r+1}))=\Tr^n_1(\lambda x^{2^r+1})=G_\lambda(x).$$
    All $\lambda \in \F_{2^{n/2}}$ are $(2^d+1)$-st powers, since $\gcd(2^d+1,2^{n/2}-1)=1$. We conclude that $F$ is plateaued with single amplitude $2^{n/2+d}$. 

    We now show that $F$ is surjective and almost balanced. Since $\gcd(2^r+1,2^{n/2}-1)=1$ we can write every element $x\in \F_{2^n}^*$ uniquely as $x=ab$ where $a \in \F_{2^{n/2}}^*$ and $b \in U=\{x \in \F_2^n \colon x^{2^{n/2}+1}=1\}$. Then 
    \[F(ab)=\Tr^n_{n/2}((ab)^{2^r+1})=a^{2^r+1}\Tr^n_{n/2}(b^{2^r+1}).\]
    Let $S = \{b \in U \colon \Tr^n_{n/2}(b^{2^r+1})=0\}$. Then $F(x)=0$ if only if $x=0$ or $x=ab$ satisfies $b \in S$. We thus have $|F^{-1}(0)|=1+(2^{n/2}-1)|S|$. For every $b \notin S$ we can choose a unique $a$ such that $F(ab)=c$ for any non-zero $c$. So we have $|F^{-1}(c)|=(2^{n/2}+1)-|S|$ for each $c \in \F_{2^{n/2}}^*$.

    We can now use Theorem~\ref{thm:AB} to determine the size of $S$. We know that $F$ has no balanced component functions (since $G$ has none as a monomial that is not a permutation) and $W_F(b,0)= \pm 2^{n/2+d}$ for any nonzero $b$. So by Theorem~\ref{thm:AB}, $F^{-1}(c)=2^{n/2}-2^d$ for any nonzero $c$ and $|F^{-1}(0)|=2^{n/2}+(2^{n/2}-1)\cdot 2^d$ (meaning that $|S|=2^d+1$).  
\end{proof}

In~\cite{carlet2015boolean}, Carlet provided the following family of Maiorana-McFarland plateaued functions with single amplitude. Now, we show that these functions are surjective and can be made balanced or almost balanced, provided that the main building blocks are suitably selected.

\begin{proposition}\label{prop: MM1}
    Let $F\colon\F_{2^m}\times\F_{2^m}\to\F_{2^m}$ defined by $F(x, y) = x\pi(y)+\phi(y)$, for $x, y \in \F_{2^m}$ be a plateaued function with  single amplitude $2^{m+1}$, where $\pi$ is 2-to-1 (that is, when every pre-image by $\pi$ has size 0 or 2) on $\F_{2^m}$ and $\phi\colon \F_{2^m} \rightarrow \F_{2^m}$ is any function. Then, the following hold:
    \begin{itemize}
        \item[1.] If $0\notin\image(\pi)$, then $F$ is balanced.
        \item[2.] If $0\in\image(\pi)$ and $\beta=\phi(y_1)=\phi(y_2)$ for $y_1,y_2\in\F_{2^m}$, s.t. $y_1\neq y_2$ and $\pi(y_1)=\pi(y_2)=0$, then $F$ is almost balanced. More precisely, $|F^{-1}(\beta)|=2^m+2(2^m-1)$ and $|F^{-1}(\alpha)|=2^m-2$ for all $\alpha \in \F_{2^m}\setminus\{ \beta \}$.
        \item[3.] If $0\in\image(\pi)$ and $\beta_1=\phi(y_1)\neq\beta_2=\phi(y_2)$ for $y_1,y_2\in\F_{2^m}$, s.t. $y_1\neq y_2$ and $\pi(y_1)=\pi(y_2)=0$, then the value distribution of $F$ is given by $|F^{-1}(\beta_1)|=|F^{-1}(\beta_2)|=2^{m+1}-2$ and $|F^{-1}(\alpha)|=2^m-2$ for all $\alpha \in \F_{2^m}\setminus\{ \beta_1,\beta_2 \}$.
    \end{itemize}
\end{proposition}
\begin{proof}
    First, we identify $\F_{2^n}$ with $\F_{2^m}\times \F_{2^m}$. Denoting $\pi^{-1}(z)=\{y\in\F_{2^m}\colon \pi(y)=z\}$, where $z\in\image(\pi)$, we compute for any non-zero $b\in\F_{2^m}$ the value
    \begin{equation*}
    \begin{split}
        |W_F(b,(0,0))|^2=&\left\lvert \sum_{x,y\in\F_{2^m}} (-1)^{\Tr(b(x\pi(y)+\phi(y)))} \right\rvert^2\\
        =&\left\lvert \sum_{y\in\F_{2^m}} (-1)^{\Tr(b\phi(y))}\sum_{x\in\F_{2^m}}(-1)^{Tr(bx\pi(y))} \right\rvert^2\\
        =&\left\lvert 2^m\sum_{\substack{y\in\F_{2^m},\\\pi(y)=0}} (-1)^{\Tr(b\phi(y))} \right\rvert^2\\
        =&\begin{cases}
	0, & 0\notin\image(\pi) \text{ or } \pi(y_1)=\pi(y_2)=0,\Tr(b(\phi(y_1)-\phi(y_2))=1,  \\
    2^{2m+2}, & \pi(y_1)=\pi(y_2)=0,\Tr(b(\phi(y_1)-\phi(y_2))=0.
\end{cases}.
    \end{split}
    \end{equation*}
\textit{Case 1.} If $0\notin\image(\pi)$, we have that $\Nb_F=0$, hence $F$ is balanced.\ \\
\textit{Case 2.} Let $y_1\neq y_2$ be two elements of $\F_{2^m}$ s.t. $\pi(y_1)=\pi(y_2)=0$ and $\phi(y_1)=\phi(y_2)=\beta\in\F_{2^m}$.  Then $\Tr(b(\phi(y_1)-\phi(y_2))=0$ for any $b \in \F_{2^m}$, so $\Nb_F=(2^m-1)\cdot2^{m+2}$, and for any $\alpha\in\F_{2^m}$ we have by Proposition~\ref{prop:cs} 
\begin{equation}\label{eq: MM case 0 inside}
    2^m-2(2^m-1) \leq |F^{-1}(\alpha)| \leq 2^m+2(2^m-1).
\end{equation}
Consider the equation $x\pi(y)+\phi(y)=\beta$. If $y\in\{y_1,y_2 \}$, then the pairs $(x,y_1)$ and $(x,y_2)$, for $x\in\F_{2^m}$, give $2^{m+1}$ solutions of this equation. Assuming now that  $y\in\F_{2^m}\setminus\{y_1,y_2\}$ is fixed, one can see that the pairs $\left((\beta+\phi(y))(\pi(y))^{-1},y\right)\in\F_{2^m}\times\F_{2^m}$ give $2^m-2$ additional solutions. In this way, we have that $|F^{-1}(\beta)| = 2^{m+1}+2^m-2$. Since the upper bound in Eq.~\eqref{eq: MM case 0 inside} is achieved with equality, by Proposition~\ref{prop:cs} we have that $|F^{-1}(\alpha)|=\frac{2^n-|F^{-1}(\beta)|}{2^m-1}=2^m-2$ for all $\alpha \in \F_{2^m}\setminus\{ \beta \}$. \ \\ 
\noindent\textit{Case 3.} Let $y_1\neq y_2$ be two elements of $\F_{2^m}$ s.t. $\pi(y_1)=\pi(y_2)=0$ and $\beta_1=\phi(y_1)\neq\phi(y_2)=\beta_2\in\F_{2^m}$. First, consider the equation $x\pi(y)+\phi(y)=\beta_i$, for a fixed $i=1,2$. If $y=y_i$, then the pairs $(x,y_i)$, for $x\in\F_{2^m}$, give $2^{m}$ solutions of this equation. If $y\neq 
 y_i$, then $x=(\beta_i+\phi(y))(\pi(y))^{-1}$ is uniquely determined by $y\in\F_{2^m}\setminus\{y_1,y_2\}$, and hence we get $2^m-2$ more solutions. Thus $|F^{-1}(\beta_1)|=|F^{-1}(\beta_2)|=2^{m+1}-2$. Assume now that $\alpha\in\F_{2^m}\setminus\{\beta_1,\beta_2\}$. If $y\in\{y_1,y_2\}$, then clearly the equation $x\pi(y)+\phi(y)=\alpha$ has no solutions for any $x\in\F_{2^m}$. If $y\in\F_{2^m}\setminus\{y_1,y_2\}$, then again $x=(\alpha+\phi(y))(\pi(y))^{-1}$ is uniquely determined by a fixed $y\in\F_{2^m}\setminus\{y_1,y_2\}$, and hence we get $2^m-2$ solutions. Thus, $|F^{-1}(\alpha)| = 2^m-2$, for all $\alpha\in\F_{2^m}\setminus\{\beta_1,\beta_2\}$.
\end{proof}
Similarly, we analyze a construction of Maiorana-McFarland plateaued functions~\cite{carlet2015boolean} with bent and semi-bent components.
\begin{proposition}\label{prop: MM2}
    Let $m$ be a positive integer and $\pi$ be a permutation of $\F_2^m$. Let $\pi(y^*)=0$. Let $i$ be an integer coprime with $m$. Then, for the function $F\colon\F_{2^m}\times\F_{2^m}\to\F_{2^m}\times\F_{2^m}$ defined by $F(x, y) = (x\pi(y), x(\pi(y))^{2^i})$ it holds that $|F^{-1}(0,0)|=2^{m+1}-1$ and $|F^{-1}(\alpha)|=1$, for any other $\alpha\in\image(F)\setminus\{(0,0)\}$.
\end{proposition}
\begin{proof}
    Consider the equation $F(x,y)=(x\pi(y) , x(\pi(y))^{2^i})=(0,0)$. If $y=y^*$, there are $2^m$ solutions $(x,y^*)$, where $x\in\F_{2^m}$. If $x=0$, then $F(0,y)=(0,0)$ has $2^m$ solutions $y\in\F_{2^m}$, hence $|F^{-1}(0,0)|=2^{m+1}-1$. Now assume $x,x' \neq 0$ and $y,y'\neq y^*$. Then $F(x,y)=F(x',y')$ if and only if 
    \[x\pi(y)=x'\pi(y') \text{ and } x\pi(y)^{2^i}=x'\pi(y')^{2^i},\]
    leading to $x'=x\pi(y)/\pi(y')$ and (plugging this into the second equation) $\pi(y)^{2^i-1}=\pi(y')^{2^i-1}$, which only has the solution $y=y'$ (leading to $x=x'$) since $x \mapsto x^{2^i-1}$ is a bijection on $\F_{2^m}$ because $i$ is coprime to $m$. So $|F^{-1}(\alpha)|=1$ for any other $\alpha\in\image(F)\setminus\{(0,0)\}$ as claimed.
\end{proof}

Finally, we note that by applying surjective linear mappings $L\colon\F_p^m\to\F_p^{k}$ to the output of considered in this section (almost) balanced plateaued functions $F\colon\F_p^n\to\F_p^m$ one can get plateaued functions $L\circ F\colon\F_p^n\to\F_p^k$ which are again (almost) balanced, since this construction preserves the (almost) balanced property and the Walsh transform, as noted in~\cite{KoelschPolujan2024}.

\section{Combinatorial structure of plateaued APN functions} \label{sec:APN}

In this section, we consider in detail the connections between imbalance and the Walsh transform of plateaued APN functions, which we further use to analyse the combinatorial structure of these functions. First, we recall several statements that will be useful for our purpose.
\begin{proposition}[{\cite[Lemma 2]{Koelsch2023}}]\label{prop:KKK}
    Let $F \colon \F_2^n \rightarrow \F_2^n$ be an APN function. Then $\Nb_F \leq 2^{n+1}-2$.
\end{proposition}

\begin{proposition}[{\cite[Proposition 9]{carlet2015boolean} and \cite[Corollary 3]{berger2006almost} \label{prop:carlet}}]
      Let $F \colon \F_2^n \rightarrow \F_2^n$ be a plateaued APN function where the amplitudes of the component function $\langle b,F \rangle$ are $2^{\lambda_b}$. Then 
      \begin{equation} \label{eq:carlet}
                \sum_{b \in \F_2^n \setminus \{0\}}2^{2\lambda_b} = 2^{n+1}(2^n-1).
      \end{equation}
      In particular, if $n$ is even, $F$ has at least $2(2^n-1)/3$ bent component functions.
\end{proposition}

In the following statement, we provide lower and upper bounds on the imbalance of plateaued APN functions and give a combinatorial characterization of functions achieving the upper and lower bounds with equality.
\begin{proposition}\label{prop: APNs with extreme imbalance}
     Let $n$ be even and $F \colon \F_2^n \rightarrow \F_2^n$ be a plateaued APN function. Then 
     \[(2^{n+1}-2)/3\leq \Nb_F \leq 2^{n+1}-2.\]
     The lower bound is satisfied with equality if and only if $F$ has $2(2^n-1)/3$ bent component functions and $(2^n-1)/3$ balanced component functions. The upper bound is satisfied with equality if and only if no component function of $F$ is balanced.
\end{proposition}
\begin{proof}
    By Proposition~\ref{prop:carlet}, $F$ has at least $2(2^n-1)/3$ bent component functions. Since bent functions cannot be balanced, we have for these components $\langle b,F \rangle$ that $W_F(b,0)^2=2^n$. Thus
    \[\Nb_F=\frac{1}{2^n}\sum_{b \in \F_2^n \setminus \{0\}}|W_F(b,0)|^2 \geq \frac{1}{2^n} \cdot \frac{2(2^n-1)}{3} \cdot 2^n=\frac{2(2^n-1)}{3}\]
    with equality if $W_F(b,0)=0$ for $(2^n-1)/3$ component functions.
    
    The upper bound is already proven in Proposition~\ref{prop:KKK}. We have $\Nb_F \leq  2^{n+1}-2$ by Proposition~\ref{prop:KKK}. Moreover, $|W_F(b,0)| \in \{0,2^{\lambda_b}\}$, so $\Nb_F$ is maximized if $|W_F(b,0)|^2=2^{2\lambda_b}$ for all $b\in \F_2^n \setminus \{0\}$. The result then follows from Proposition~\ref{prop:carlet}.
\end{proof}

In the following statement, we derive some useful results on the imbalance of plateaued APN functions.
\begin{lemma} \label{lem:balanced}
    Let $n$ be even and $F \colon \F_2^n \rightarrow \F_2^n$ be a plateaued APN function. Then $\Nb_F \equiv 2 \pmod 4$. In particular, if $F$ has at least one balanced component, then $\Nb_F \leq 2^{n+1}-6.$
\end{lemma}
\begin{proof}
     Denote again the  amplitudes of the component functions $\langle b,F \rangle$ by $2^{\lambda_b}$, so $|W_F(b,0)| \in \{0,2^{\lambda_b}\}$. Let $S \subset \F_2^n \setminus \{0\}$ be the set of balanced components, i.e., $b \in S$ if and only if $\langle b,F \rangle$ is balanced. 
     
     A balanced component function has an amplitude of at least $2^{n/2+1}$ since bent functions cannot be balanced. So, using Proposition~\ref{prop:carlet},
    $$\Nb_F \leq 2^{n+1}-2-\frac{1}{2^n}\sum_{b \in S}2^{2\lambda_b}=2^{n+1}-2-\sum_{b \in S}2^{2\lambda_b-n}.$$
    We have $2\lambda_b\geq n+2$, so the result follows since $\sum_{b \in S}2^{2\lambda_b-n} \equiv 0 \pmod 4$.
\end{proof}
With this technical result, we further refine the number of bent components of plateaued APN functions.
\begin{proposition}
      Let $n$ be even and $F \colon \F_2^n \rightarrow \F_2^n$ be a plateaued APN function. The number of bent components of $F$ is $2 \pmod 4$.
\end{proposition}
\begin{proof}
    We have by definition
    \begin{equation}
    \label{eq:def}
        2^n\Nb_F = \sum_{\beta \in \F_2^n\setminus \{0\}} W_F(b,0)^2.
    \end{equation}
    
    By Lemma~\ref{lem:balanced}, we have that $\Nb_F \equiv 2 \pmod 4$, so taking Eq.~\eqref{eq:def} modulo $2^{n+2}$ yields
    \[2^{n+1} \equiv 2^n\mathcal{B}(F) \pmod{2^{n+2}},\]
    where $\mathcal{B}(F)$ denotes the number of bent components functions of $F$. The result follows immediately.
\end{proof}
Recall the following classification of value distributions of APN functions with the minimum image set size.
\begin{theorem}{\cite[Theorem 3]{Koelsch2023}} \label{thm:KKK}
     Let $n$ be even and $F \colon \F_2^n \rightarrow \F_2^n$ be an APN function. Then $|\image(F)| \geq \frac{2^n+2}{3}$. If equality holds then $F$ has one of the following value distributions:
     \begin{enumerate}
         \item $X_1=1$, $X_2=\dots=X_{(2^n+2)/3}=3$,
         \item $X_1=X_2=2$, $X_3=\dots=X_{(2^n+2)/3}=3$,
         \item $X_1=X_2=X_3=2$, $X_4=\dots=X_{(2^n-1)/3}=3$, $X_{(2^n+2)/3}=4$.
     \end{enumerate}
\end{theorem}
Several examples of APN functions with the first value distribution are well-known; these are exactly $3$-to-$1$ functions. The existence of APN functions with the other two preimage distributions is an open problem (see~\cite[Open problem 1]{Koelsch2023}). Using our tools, we can show that there are no \emph{plateaued} APN functions with the second value distribution.

\begin{theorem}\label{prop: 2nd impossible}
     Let $n$ be even and $F \colon \F_2^n \rightarrow \F_2^n$ be a plateaued APN function. Then $F$ does not have the second value distribution listed in Theorem~\ref{thm:KKK}.
\end{theorem}
\begin{proof}
    Assume $F$ has the second value distribution listed in Theorem~\ref{thm:KKK}. Then
     \[\sum_{\beta \in \F_{2}^m} |F^{-1}(\beta)|^2=2\cdot 4 + \frac{2^n-4}{3} \cdot 9=3\cdot 2^n-4.\]
     By Proposition~\ref{prop:starting}, we have 
     \[\sum_{\beta \in \F_{2}^m} |F^{-1}(\beta)|^2=2^n+\Nb_F,\]
     so $\Nb_F=2^{n+1}-4$. This is a contradiction to Lemma~\ref{lem:balanced}.
\end{proof}

We can also show that plateaued APN functions with minimal image set size cannot have any balanced component functions.
\begin{theorem} \label{thm:nobalanced}
    Let $n$ be even and $F \colon \F_2^n \rightarrow \F_2^n$ be a plateaued APN function. If $|\image(F)|=\frac{2^n+2}{3}$, then $F$ has no balanced component functions.
\end{theorem}
\begin{proof}
    By Proposition~\ref{prop:imageset} we have 
    \[|\image(F)| \geq \frac{2^{2n}}{2^n+\Nb_F}.\]
    If $F$ has at least one balanced component, we can bound $\Nb_F\leq 2^{n+1}-6$ by Lemma~\ref{lem:balanced} and
     \[|\image(F)| \geq \frac{2^{2n}}{3\cdot 2^n-6}.\]
    It remains to show that $\frac{2^{2n}}{3\cdot 2^n-6}>\frac{2^n+2}{3}$. This is verified using elementary means: After multiplying by the denominators, we have
    \[3\cdot 2^{2n}>(2^n+2)(3\cdot 2^n-6)=3\cdot 2^{2n}-12,\]
    which is clearly true.
\end{proof}

\section{Conclusion}\label{sec: concl}

In this paper, we provided more insights into the possible combinatorial structure of plateaued functions, particularly APN functions, and developed further techniques for the analysis of value distributions of vectorial functions. We believe that using these tools it should be possible to answer the following questions, thus contributing to a deeper understanding of various classes of cryptographically significant functions. 
\begin{itemize}
    \item[1.] Value distributions of APN functions with the minimum image set size have been considered in detail in~\cite{Koelsch2023}. As mentioned in Theorem~\ref{thm:KKK}, such APN functions can have only 3 possible value distributions. Within the class of plateaued APN functions with the minimum image set size, many constructions have the first value distribution (in the sense of Theorem~\ref{thm:KKK}), and the second value distribution is impossible, as we showed in Theorem~\ref{prop: 2nd impossible}. Given these observations, it is natural to ask whether the third value distribution is possible for plateaued APN functions with the minimum image set size.
    \item[2.] In~\cite{PottPasEt}, it was shown that the maximum number of bent component functions of a vectorial function $F \colon \F_2^n \to \F_2^n$ is $2^{n} - 2^{n/2}$. Later, in~\cite{Mesnager2019}, it was proven that APN plateaued functions cannot have the maximum number of bent components. The best current bound on the number of bent components of a plateaued APN functions was found recently in~\cite{Koelsch2025}. It is however unclear if these bounds are generally tight.  We believe that providing better bounds on the number of bent components of plateaued APN functions is an interesting and important research direction.
    \item[3.] Generally, provide more primary and secondary constructions of (vectorial) plateaued functions, especially with single amplitude and special value distributions.
\end{itemize}
  
\section*{Acknowledgments} We thank the anonymous reviewers for their detailed feedback, spotting several errors and typos. In particular, we thank one reviewer for providing an idea for a simpler proof of Theorem 4.1.


\appendix
\section{Appendix: Proof of Lemma~\ref{lem:podd}}\label{sec: Appendix}
    \begin{proof}[Proof of Lemma~\ref{lem:podd}]
          Let us now show that $W_F(b,0) \in \Z$. So assume $p>2$. Since $F$ is $d$-to-1, $F$ has no balanced components (see the proof of Theorem~\ref{thm:platdto1}). We then have by~\cite[Theorem 2]{Hyun16} that for any nonzero $b$ that $W_F(b,0)=\epsilon p^{(n+t)/2}\zeta_p^r$, $\epsilon \in \{-1,1\}$ if $n+t$ is even or $p\equiv 1 \pmod 4$ and $\epsilon \in \{-i,i\}$, if $n+t$ is odd and $p \equiv 3 \pmod 4$.  Write (using the standard integral basis of roots of unity for the cyclotomic field) $C = \sum_{i=1}^{p-1} a_i \zeta_p^i$ with $a_i \in \Z$. Then
    \begin{equation} \label{eq:walshp}
    W_F(b,0)=\epsilon p^{(n+t)/2}\zeta_p^r=d\cdot C+1 = \sum_{i=1}^{p-1} (da_i-1) \zeta_p^i,
    \end{equation}
    where we use that the sum of all roots of unity is $0$.
    If $n+t$ is even then $\epsilon p^{(n+t)/2} \in \Z$. So if $r\neq 0$, by the properties of the integral basis,  we have $da_i=1$ for all $i \neq r$, which contradicts $d>1$ and $a_i \in \Z$. So $r=0$ and $W_F(b,0) \in \Z$.
    
    Now consider the case where $n+t$ is odd. Recall that (see e.g.~\cite{Nyberg91}), 
    $ \epsilon \sqrt{p} = \sum_{j=1}^{p-1}\left( \frac{j}{p}\right) \zeta_p^j$, where $\left( \frac{j}{p}\right)$ is the Legendre symbol. Then 
    \[\epsilon p^{(n+t)/2}\zeta_p^r = p^{(n+t-1)/2} \epsilon \sqrt{p} \zeta_p^r=p^{(n+t-1)/2}\sum_{j=1}^{p-1}\left( \frac{j}{p}\right) \zeta_p^{j+r},\]
    so combining this with Eq.~\eqref{eq:walshp} gives
    \[\sum_{j=1}^{p-1}\left(p^{(n+t-1)/2}\cdot \left( \frac{j}{p}\right)\right) \zeta_p^{j+r} = \sum_{i=1}^{p-1} (da_i-1) \zeta_p^i.\]
    If $r=0$ this implies that $da_i = p^{(n+t-1)/2}\left( \frac{i}{p}\right)+1$ for all $i$, meaning that $d$ divides both $p^{(n+t-1)/2}+1$ and $p^{(n+t-1)/2}-1$, so $d|2$, contradicting $d>2$. If $r\neq 0$ then 
    \begin{align*}
        p^{(n+t-1)/2}\left(\sum_{\substack{j=1\\j \neq p-r}}^{p-1} \left( \frac{j}{p}\right) \zeta_p^{j+r} + \left(\frac{p-r}{p}\right)\right) &= 
         p^{(n+t-1)/2}\left(\sum_{\substack{j=1\\j \neq r}}^{p-1} \left( \frac{j-r}{p}\right) \zeta_p^{j} - \sum_{j=1}^{p-1}\left(\frac{p-r}{p}\right) \zeta_p^j\right)\\ &= 
        \sum_{j=1}^{p-1} (da_j-1) \zeta_p^j.
    \end{align*}
    So if $\left( \frac{j-r}{p}\right) \neq \left(\frac{p-r}{p}\right)$, (which is guaranteed to happen by the properties of the Legendre symbol), we have $da_j-1=0$, contradicting $d>1$. We conclude that $W_F(b,0) \in \Z$.
    \end{proof}
\end{document}